\immediate\write18{biber \jobname}

\documentclass{amsart}

\usepackage{microtype}
\usepackage{amssymb}
\usepackage{mathtools}
\usepackage{tikz-cd}
\usepackage{mathbbol} 

\usepackage[
	backend=biber,
	style=alphabetic,
	backref=true,
	url=false,
	doi=false,
	isbn=false,
	eprint=false]{biblatex}

\renewbibmacro{in:}{}  
\AtEveryBibitem{\clearfield{pages}}  

\DeclareFieldFormat{title}{\myhref{\mkbibemph{#1}}}
\DeclareFieldFormat
[article,inbook,incollection,inproceedings,patent,thesis,unpublished]
{title}{\myhref{\mkbibquote{#1\isdot}}}

\newcommand{\doiorurl}{%
	\iffieldundef{url}
	{\iffieldundef{eprint}
		{}
		{http://arxiv.org/abs/\strfield{eprint}}}
	{\strfield{url}}%
}

\newcommand{\myhref}[1]{%
	\ifboolexpr{%
		test {\ifhyperref}
		and
		not test {\iftoggle{bbx:eprint}}
		and
		not test {\iftoggle{bbx:url}}
	}
	{\href{\doiorurl}{#1}}
	{#1}%
}

\usepackage[
	bookmarks=true,
	linktocpage=true,
	bookmarksnumbered=true,
	breaklinks=true,
	pdfstartview=FitH,
	hyperfigures=false,
	plainpages=false,
	naturalnames=true,
	colorlinks=true,
	pagebackref=false,
	pdfpagelabels]{hyperref}

\hypersetup{
	colorlinks,
	citecolor=blue,
	filecolor=blue,
	linkcolor=blue,
	urlcolor=blue
}

\usepackage[capitalize, noabbrev]{cleveref}
\crefname{subsection}{\S\!\!}{subsections}
\crefname{subsubsection}{\S\!\!}{subsubsections}

\calclayout

\makeatletter
\@namedef{subjclassname@2020}{%
	\textup{2020} Mathematics Subject Classification}
\makeatother

\newtheorem*{theorem*}{Theorem}

\newtheorem*{proposition*}{Proposition}
\newtheorem{lemma}[subsubsection]{Lemma}
\newtheorem*{lemma*}{Lemma}

\newtheorem*{corollary*}{Corollary}

\theoremstyle{definition}

\newtheorem*{definition*}{Definition}

\newtheorem*{remark*}{Remark}

\newtheorem*{example*}{Example}

\newtheorem*{construction*}{Construction}

\newtheorem*{convention*}{Convention}

\newtheorem*{terminology*}{Terminology}

\newtheorem*{notation*}{Notation}

\newtheorem*{question*}{Question}

\DeclareMathOperator{\face}{d}
\DeclareMathOperator{\dege}{s}
\DeclareMathOperator{\bd}{\partial}

\newcommand{\ot}{\otimes}
\DeclareMathOperator{\EZ}{EZ}
\DeclareMathOperator{\AW}{AW}

\newcommand{\N}{\mathbb{N}}
\newcommand{\Z}{\mathbb{Z}}

\newcommand{\sym}{\mathbb{S}}
\newcommand{\cyc}{\mathbb{C}}
\newcommand{\Ftwo}{{\mathbb{F}_2}}
\newcommand{\Fp}{{\mathbb{F}_p}}

\newcommand{\Set}{\mathsf{Set}}

\newcommand{\Ch}{\mathsf{Ch}}

\newcommand{\sSet}{\mathsf{sSet}}

\DeclareMathOperator{\Sq}{Sq}

\DeclareMathOperator{\chains}{N}
\DeclareMathOperator{\cochains}{N^{\vee}}

\DeclarePairedDelimiter\bars{\lvert}{\rvert}

\DeclarePairedDelimiter\set{\{}{\}}

\DeclarePairedDelimiter\floor{\lfloor}{\rfloor}

\newcommand{\id}{\mathsf{id}}

\newcommand{\Hom}{\mathrm{Hom}}
\newcommand{\End}{\mathrm{End}}

\newcommand{\xla}[1]{\xleftarrow{#1}}
\newcommand{\xra}[1]{\xrightarrow{#1}}
\newcommand{\defeq}{\stackrel{\mathrm{def}}{=}}

\newcommand{\rD}{\mathrm{D}}
\newcommand{\rE}{\mathrm{E}}

\newcommand{\rH}{\mathrm{H}}

\newcommand{\rP}{\mathrm{P}}

\newcommand{\cC}{\mathcal{C}}

\newcommand{\cE}{\mathcal{E}}

\newcommand{\cO}{\mathcal{O}}
\newcommand{\cP}{\mathcal{P}}

\newcommand{\cR}{\mathcal{R}}
\newcommand{\cS}{\mathcal{S}}

\newcommand{\cW}{\mathcal{W}}

\newcommand{\bcirc}{\bm{\circ}}
\DeclareMathOperator{\EE}{E}
\DeclareMathOperator{\Shi}{Shi}
\newcommand{\BE}{\cE}
\newcommand{\cp}{\smallsmile}
\newcommand{\pdfC}{\texorpdfstring{$\cyc$}{C}}

\DeclareMathOperator{\card}{card}                   
\usepackage{bm}

\title[An effective proof of the Cartan formula: Odd primes]{An effective proof of the Cartan formula:\\Odd primes}

\author[Medina-Mardones]{Anibal~M.~Medina-Mardones}
\address{LAGA, Universit\'e Sorbonne Paris Nord.
	 Institut Galilée 99, avenue Jean-Baptiste Clément, 93430, Villetaneuse, France.}
\email{\href{mailto:medina-mardones@math.univ-paris13.fr}{medina-mardones@math.univ-paris13.fr}}

\author[Cantero-Mor\'an]{Federico~Cantero-Mor\'an}
\address{Departamento de Matem\'aticas, Universidad Aut\'onoma de
	Madrid \& ICMAT. Calle Francisco Tom\'as y Valiente, 7. 28049, Madrid, Spain.}
\email{\href{mailto:federico.cantero@uam.es}{federico.cantero@uam.es}}

\date{\today}

\begin{document}

\begin{abstract}
	The Cartan formula relates the cup product and the action of the Steenrod algebra on mod~$p$ cohomology.
	For any pair of mod~$p$ cocycles in a simplicial set, where $p$ is an odd prime, we effectively construct a natural coboundary descending to this formula in cohomology.
\end{abstract}
	\maketitle

\section{Introduction} \label{s:introduction}

Steenrod introduced in \cite{steenrod1947products} his celebrated squares
\begin{equation*}
	\Sq^k \colon \rH^\vee(X; \Ftwo) \to \rH^\vee(X; \Ftwo)
\end{equation*}
on the mod 2 cohomology of simplicial sets via an effective construction of cup-$i$ products on their integral cochains
\[
\cp_i \colon \cochains(X) \ot \cochains(X) \to \cochains(X),
\]
which are degree $i$ products satisfying:
\[
\Sq^{i-\bars{\alpha}}[\alpha] \defeq [\alpha \cp_i \alpha],
\]
for $\alpha$ a mod $2$ cocycle and $[\alpha]$ its represented cohomology class.
The cup-$0$ product induces the cup product in cohomology
\[
[\alpha] \cdot [\beta] \defeq [\alpha \cp_0 \beta],
\]
whose interaction with the Steenrod squares is given by the Cartan formula:
\[
\Sq^k\!\big([\alpha] \cdot [\beta]\big) = \sum_{i+j=k} \Sq^i [\alpha] \cdot \Sq^j[\beta].
\]
This relation in cohomology implies the existence, for any pair of mod $2$ cocycles, of a coboundary $\delta\zeta_i(\alpha,\beta)$ satisfying:
\begin{equation}\label{e:lift cartan even}
	\delta\zeta_i(\alpha, \beta) =
	(\alpha \smallsmile_0 \beta) \cp_i (\alpha \cp_0 \beta) \ +
	\sum_{i=j+k} (\alpha \smallsmile_j \alpha) \cp_0 (\beta \cp_k \beta).
\end{equation}

Inspired by the application of cup-$i$ products in the theory of topological phases (refer to \cite{kapustin2015cobordism, gaiotto2016spin, kapustin2017fermionic} for examples), Kapustin inquired about the construction of natural cochains $\zeta_i(\alpha, \beta)$ as in \eqref{e:lift cartan even}.
This was presented in \cite{medina2020cartan} and later used in \cite{barkeshli2021classification} to classify fermionic $(2+1)$-dimensional invertible phases.

Not long after Steenrod's construction was introduced, a non-effective perspective based on group homology was developed, which allowed the definition not only of Steenrod squares but also of Steenrod operations
\begin{align*}
	P^s \colon& \rH^\vee(X; \Fp) \to \rH^\vee(X; \Fp), \\
	\beta P^s \colon& \rH^\vee(X; \Fp) \to \rH^\vee(X; \Fp),
\end{align*}
for odd primes.
The interaction of these with the cup product are also controlled by Cartan formulas:
\begin{align*}
	\rP_s\big([a][b]\big) =&
	\sum_{i+j=s} \rP_i[a] \, \rP_j[b], \\
	\beta\rP_s\big([a][b]\big) =&
	\sum_{i+j=s} \beta\rP_{i+1}[a] \, \rP_j[b] \ +\ (-1)^{\bars{a}} \rP_i[a] \, \beta\rP_{j+1}[b].
\end{align*}

By extending the concept of cup-$i$ products to other primes via the introduction of cup-$(p,i)$ products, an effective construction of Steenrod operations was introduced in \cite{medina2021may_st} and implemented in the specialized computer algebra system \texttt{ComCH} \cite{medina2021comch}.
The goal of this paper is to utilize these cup-$(p,i)$ products to lift the Cartan formulas to the cochain level, and to effectively construct natural coboundaries that generalize $\zeta_i(a,b)$ in \cref{e:lift cartan even} for odd primes.
Our method for constructing these odd Cartan coboundaries is solely based on the combinatorial structure of simplices, which enables a computer to carry out the process locally, focusing on individual simplices sequentially.
An actual implementation is left to future work.

\subsection*{Outline}

In \cref{s:preliminaries} we recall the Eilenberg--Zilber contraction and introduce the basic reordering, a shuffle permutation in the symmetric group $\sym_{2r}$ for any $r$.
We present in \cref{s:may_st} a summary of \cite{medina2021may_st}, which makes explicit the definitions of Steenrod operations \cite{steenrod1953cyclic} from the general viewpoint of May \cite{may1970general}.
\cref{s:cartan} is devoted to introducing the notions of Cartan relator and Cartan coboundary, whereas \cref{s:effective} to the effective construction of these for the cochains of spaces.
We present in \cref{s:postponed} a relatively long proof postponed from the previous section.

\subsection*{Acknowledgment}

We would like to thank Greg Brumfiel, John Morgan, Andy Putmam, and Dennis Sullivan for their insights, question, and comments related to this project.
A.M-M. acknowledges the support and excellent working conditions of the Max Planck Institute for Mathematics in Bonn.
F.C. was supported by grants SI3/PJI/2021-00505 from Comunidad de Madrid and PID2019-108936GB-C21 from the Spanish government.

\section{Preliminaries}\label{s:preliminaries}

Throughout this work $p$ will denote an odd prime and $\Fp$ the field with $p$ elements.
We will assume the ground ring to be $\Z$ unless stated otherwise.
We use homological grading, with the linear dual of an element in degree $d$ placed in degree $-d$.

\subsection{Eilenberg--Zilber contraction}

The functor of (normalized) chains $\chains \colon \sSet \to \Ch$ from simplicial sets to chain complexes does not preserve products, but there are natural chain maps
\[
\begin{tikzcd}
	\chains(X \times Y) \arrow[r,shift left=2pt,"\AW"] &
	\chains(X) \ot \chains(Y) \arrow[l,shift left=2pt,"\EZ"]
\end{tikzcd}
\]
and a natural linear map
\[
\Shi \colon \chains(X \times X) \to \chains(X \times X)
\]
satisfying
\[
\EZ \circ \AW = \id, \qquad
\AW \circ \EZ = \id + \bd \circ \Shi + \Shi \circ \bd.
\]
Closed form formulas describing these maps can be found for example in \cite[56]{real2000homological}.

Let $\rD \colon X \to X \times X$ be the diagonal map of simplicial sets.
We denote the compositions $\AW \circ \chains\rD$ by $\Delta_{\AW}$.

\subsection{The basic reordering}\label{ss:reordering}

For any $r \in \N$, the element $\tau_r \in \sym_{2r}$ is the shuffle permutation mapping the first and second ``decks'' to odd and even integers respectively.
Explicitly, for $\ell \in \{1,\dots,2r\}$ we have
\begin{equation*}
	\tau_r(\ell) =
	\begin{cases}
		2\ell-1 & \ell \leq r, \\
		2(\ell-r) & \ell > r.
	\end{cases}
\end{equation*}
When $r$ is clear from the context we will omit it from the notation.
Please notice that in a graded module we have
\[
\tau(a^{\ot r} \ot b^{\ot r}) =
(-1)^{\frac{(r-1)r}{2}\bars{a}\bars{b}} \, (a \ot b)^{\ot r}.
\]

\section{May--Steenrod structures}\label{s:may_st}

We now present a summary of \cite{medina2021may_st}, which makes explicit the definitions of \cite{steenrod1953cyclic} from the general viewpoint of \cite{may1970general}.
We assume familiarity with the notion of operad as presented for example in \cite{may1997operads} or \cite{loday2012operads}.

\subsection{\pdfC-modules}

For $r > 0$, we denote by $\cyc_r$ the cyclic group of order $r$ thought of as the subgroup of the symmetric group $\sym_r$ generated by the cycle permutation $\rho = (1,2,\dots,r)$.
When convenient, we will simplify our notation using the bijection
\[
\begin{tikzcd}[column sep=small,row sep=-5]
	\{0,\dots,r-1\} \rar["\cong\ "] & \cyc_r \\
	\qquad i \rar[maps to] & \rho^i.
\end{tikzcd}
\]
A \textit{$\cyc$-module} $\cP$ is a set $\{\cP(r)\}_{r>0}$ with $\cP(r)$ a chain complex with a right action of $\cyc_r$.
A \textit{$\cyc$-equivariant map} $\phi \colon \cP_1 \to \cP_2$ is a set $\set{\phi(r) \colon \cP_1(r) \to \cP_2(r)}_{r>0}$ with $\phi(r)$ a $\cyc_r$-equivariant map.
We say $f$ is a chain map or a quasi-isomorphism if each $\phi(r)$ is.
A similar convention is applied for chain homotopies.
When convenient we will abuse notation and denote $\phi(r)$ simply as $\phi$.
There are evident forgetful functors to the category of $\cyc$-modules from those of $\sym$-modules and operads, which we will use without further notice.

\subsection{Minimal cyclic resolution}

By $\cW(r)$ we denote the minimal free resolution of $\Z$ by $\Z[\cyc_r]$-modules
\begin{equation}\label{eq: minimal resolution}
	\Z[\cyc_r]\{e_0\} \xla{T} \Z[\cyc_r]\{e_1\} \xla{N} \Z[\cyc_r]\{e_2\} \xla{T} \cdots
\end{equation}
where
\begin{equation} \label{eq: T and R definition}
	\begin{split}
		T &= \rho - 1, \\
		N &= 1 + \rho + \cdots + \rho^{r-1}.
	\end{split}
\end{equation}
It is equipped with the structure of a $\cyc_r$-equivariant coalgebra defined by:
\begin{align*}
	\varepsilon(e_0) &= 1, \\
	\Delta(e_{2i}) &=
	\sum_{i=j+k} e_{2j} \ot e_{2k} \ + \! \sum_{i-1=j+k} \ \sum_{0 \leq s < t < r} \rho^s e_{2j+1} \ot \rho^t e_{2k+1}, \\
	\Delta(e_{2i+1}) &=
	\sum_{i=j+k} e_{2j} \ot e_{2k+1} \ +\ e_{2j+1} \ot \rho e_{2k},
\end{align*}
where $\Z$ is acted on trivially and $\cW(r) \ot \cW(r)$ diagonally.
We denote by $\cW$ the $\cyc$-module defined by $\set{\cW(r)}_{r > 0}$ which is a coalgebra in the category of $\cyc$-modules.

\subsection{May--Steenrod structures}

Recall that an $E_\infty$-algebra structure on a chain complex $A$ is an operad morphism to $\End(A) = \{\Hom(A^{\ot r}, A)\}_{r>0}$ from an operad $\cR$ such that, for each $r > 0$, the chain complex $\cR(r)$ is contractible and its $\sym_r$-action is free.

A \textit{May--Steenrod structure} on an chain complex $A$ is an $E_\infty$-structure $\phi \colon \cR \to \End(A)$ on $A$ together with a $\cyc$-equivariant quasi-isomorphism $\iota \colon \cW \to \cR$.
We write $\psi$ for $\phi \circ \iota$ and abusively use it to denote the May--Steenrod structure.
We refer to a chain complex with a May--Steenrod structure as a \textit{May--Steenrod complex}.

A May--Steenrod complex $(A,\psi)$ has a preferred product $\psi(2)(e_0) \colon A \ot A \to A$ which we denote for elements $a, b \in A$ simply as $a \cp b$.
Generalizing this, we refer to $\psi(r)(e_i) \colon A^{\ot r} \to A$ as the \textit{cup-$(r,i)$ product}, denoting it $\psi_i^r$ or simply $\psi_i$ when $r$ is clear from the context.

\subsection{Homology operations}

Let $(A,\psi)$ be a May--Steenrod complex.
After extending scalars to $\Fp$ the assignment sending $a$ to $\psi_i(a^{\ot p})$ sends cycles to cycles and induces a linear map in homology.
We denote its induced map on the mod $p$ homology of $A$ by
\begin{equation}\label{e:D_i^p}
	\rD_i^p \colon \rH_n(A; \Fp) \to \rH_{np+i}(A; \Fp).
\end{equation}
We will emphasize these homology operations as opposed to the operations $P_s$ and $\beta P_s$ that we now define.
The map $\rD_i^p$ in \eqref{e:D_i^p} is $0$ unless $n$ is even (resp. odd) and $i+\varepsilon$ is an even (resp. odd) multiple of $(p-1)$ for some $\varepsilon \in \{0,1\}$.
Please consult \cite[Proposition 2.3. (iv)]{may1970general} for a proof.
With this in mind, one defines the \textit{Steenrod operations}
\begin{align*}
	P_s \colon& H_\bullet(A; \Fp) \to H_{\bullet + 2s(p-1)}(A; \Fp), \\
	\beta P_s \colon& H_\bullet(A; \Fp) \to H_{\bullet + 2s(p-1) - 1}(A; \Fp),
\end{align*}
by sending the class represented by a cycle $a \in A \ot \Fp$ of degree $n$ to the classes represented respectively for $\varepsilon \in \{0,1\}$ by
\begin{equation*}
	(-1)^s \nu(n) \, \rD^p_{(2s-n)(p-1)-\varepsilon}[a]
\end{equation*}
where $\nu(n) = (-1)^{n(n-1)m/2}(m!)^n$ and $m = \floor{p/2} = (p-1)/2$.
The constant $\nu(n)$ and notation $\beta P_s$ are motivated respectively by the unstable $\rD_{n(p-1)}^p[\alpha] = \nu(n)[\alpha]$ and Bockstein relations $\beta P_s = \beta \circ P_s$ respectively.
The first holding in the cohomology of spaces and the second in general.

\section{The Cartan formula}\label{s:cartan}

\subsection{Cartan relators}

Given an $\sym$-module $\cS$ we denote by $\cS^\wedge$ the $\cyc$-module $\cS^\wedge(r) = \cS(2r)$ with
action defined by
\begin{equation}\label{eq:cyclic to 2-symetric}
	\cyc_r \xra{\rD}
	\cyc_r \times \cyc_r \hookrightarrow
	\sym_r \times \sym_r \xra{\mathrm{split}}
	\sym_{2r}.
\end{equation}
For a May--Steenrod complex $(A,\psi)$ we define two $\cyc$-equivariant chain maps
\[
F_\psi,\, G_\psi \colon \cW \to \End(A)^\wedge
\]
as follows.
Let us denote the operad $\End(A)$ by $\cO$.
Then,
\begin{align*}
	F_\psi \colon& \cW(r) \xra{\psi(r)} \cO(r) \xra{\id\, \ot \psi_0^{\ot r}}
	\cO(r) \ot \cO(2)^{\ot r} \xra{\bcirc_{\cO}}
	\cO(2r), \\
	G_\psi \colon& \cW(r) \xra{\Delta}
	\cW(r)^{\ot 2} \xra{\psi^{\ot 2}}
	\cO(r)^{\ot 2} \xra{\psi_0 \ot\, \id}
	\cO(2) \ot \cO(r)^{\ot 2} \xra{\bcirc_{\cO}}
	\cO(2r).
\end{align*}

\begin{definition*}
	A \textit{Cartan relator} for a May--Steenrod complex $(A, \psi)$ is a $\cyc$-equivariant chain homotopy between $\tau F_\psi$ and $G_\psi$ where $\tau$ is the canonical reordering of \cref{ss:reordering}.
\end{definition*}

\subsection{Cartan boundary constructions}\label{ss:cartan_coboundary}

As we will see, the existence of a Cartan relator implies the Cartan formulas:
\begin{align*}
	\rP_s\big([a][b]\big) =&
	\sum_{i+j=s} \rP_i[a] \, \rP_j[b], \\
	\beta\rP_s\big([a][b]\big) =&
	\sum_{i+j=s} \beta\rP_{i+1}[a] \, \rP_j[b] \ +\ (-1)^{\bars{a}} \rP_i[a] \, \beta\rP_{j+1}[b],
\end{align*}
holding for each integer $s$ and mod $p$ cycles $a$ and $b$.
Please observe that the Cartan formulas are equivalent to the following equations:
\begin{align*}
	0 &= (-1)^{\floor{p/2}\bars{a}\bars{b}} \, \rD_{2i}\big([a][b]\big) \,-\!
	\sum_{i=j+k} \rD_{2j}[a] \, \rD_{2k}[b], \\
	0 &= (-1)^{\floor{p/2}\bars{a}\bars{b}} \, \rD_{2i+1}\big([a][b]\big) \,-\!
	\sum_{i=j+k} \rD_{2j+1}[a] \, \rD_{2k}[b] \ +\ (-1)^{\bars{a}}\rD_{2j}[a] \, \rD_{2k+1}[b],
\end{align*}
ranging over $i \in \N$.
A lift of the right hand side of these to the chain level is given by the following formulas:
\begin{align*}
	&C_\psi^p(2i)(a,b) \defeq (-1)^{\floor{p/2}\bars{a}\bars{b}} \psi_{2i}\big((a \cp b)^{\ot p}\big) \ -
	\sum_{i=j+k}\big(\psi_{2j}(a^{\ot p})\big) \cp \big(\psi_{2k}(b^{\ot p})\big), \\
	&C_\psi^p(2i+1)(a,b) \defeq (-1)^{\floor{p/2}\bars{a}\bars{b}} \psi_{2i+1}\big((a \cp b)^{\ot p}\big) \\
	&\quad -\sum_{i=j+k} \big(\psi_{2j+1}(a^{\ot p})\big) \cp \big(\psi_{2k}(b^{\ot p})\big)\ +\
	(-1)^{\bars{a}}\big(\psi_{2j}(a^{\ot p})\big) \cp \big(\psi_{2k+1}(b^{\ot p})\big).
\end{align*}

\begin{definition*}
	Let $(A,\psi)$ be a May--Steenrod complex.
	A \textit{Cartan $i$-boundary construction} is a map $\zeta_i$ from pairs of mod~$p$ cycles to $A$, natural with respect to $E_\infty$-algebra chain maps, such that
	\[
	\bd \zeta_i(a,b) = C_\psi^p(i)(a, b)
	\]
	in $A \ot \Fp$.
\end{definition*}

\begin{theorem*}
	If $H_\psi$ is a Cartan relator for $(A, \psi)$ then
	\[
	\zeta_i(a,b) \defeq H_\psi^p(e_i)(a^{\ot p} \ot b^{\ot p})
	\]
	defines a Cartan $i$-boundary construction.
\end{theorem*}

\begin{proof}
	First, please observe that
	\[
	C^p_\psi(i)(a,b) =
	\big(\tau F_\psi^p - G_\psi^p\big)(e_i)(a^{\ot p} \ot b^{\ot p}).
	\]
	Second, using that in $A \ot \Fp$
	\begin{align*}
		H_\psi^p(\bd e_{2i})(a^{\ot p} \ot b^{\ot p}) =&\,
		(\rho - 1) H_\psi^p(e_{2i})(a^{\ot p} \ot b^{\ot p}) = 0, \\
		H_\psi^p(\bd e_{2i+1})(a^{\ot p} \ot b^{\ot p}) =&\,
		(1+\rho+\dots+\rho^{p-1}) H_\psi^p(e_{2i+1})(a^{\ot p} \ot b^{\ot p}) = 0,
	\end{align*}
	we have
	\begin{align*}
		\bd_A \zeta_i(a,b) =&
		\bd_A H_\psi^p(e_i)(a^{\ot p} \ot b^{\ot p}) \\ =&
		\bd_{\End(A)} H_\psi^p(e_i)(a^{\ot p} \ot b^{\ot p}) \\ =& \,
		\big(\tau F(e_i) - G(e_i)\big)(a^{\ot p} \ot b^{\ot p}) -
		H_\psi^p(\bd e_i)(a^{\ot p} \ot b^{\ot p}) \\ =&\,
		C^p_\psi(i)(a,b).\qedhere
	\end{align*}
\end{proof}

We now turn to the question of effectively constructing a natural one for the cochains of spaces.

\section{Effective constructions}\label{s:effective}

In this section we effectively construct natural Cartan coboundaries for the cochains of spaces and, more generally, for $E_\infty$-algebras defined using the Barratt--Eccles operad.

\subsection{Barratt--Eccles operad}

We review the main constructions introduced by Berger and Fresse in \cite{berger2004combinatorial}.

Let $\Set$ and $\sSet$ denote the categories of sets and simplicial sets respectively.
The symmetric monoidal functor $\EE \colon \Set \to \sSet$ is defined on objects by $\EE X_n = X^{n+1}$ with
\[
\begin{split}
	\face_i(x_0,\dots,x_n) &= (x_0,\dots,\widehat{x}_i,\dots,x_n), \\
	\dege_i(x_0,\dots,x_n) &= (x_0, \dots, x_i, x_i, \dots, x_n),
\end{split}
\]
and on morphisms by $\EE f_n = f^{\times n}$.
We notice that if $X$ is equipped with a group action then $\EE X$ is as well with
\[
(x_0,\dots,x_n) \cdot g = (x_0 \cdot g, \dots, x_n \cdot g).
\]

The usual block permutation map
\[
\bcirc_{\sym} \colon \sym_r \times \sym_{s_1} \times \cdots \times \sym_{s_r} \to \sym_{s_1+\dots+s_r}
\]
provides $\sym = \{\sym_r\}_{r>0}$ with the structure of an operad in $\Set$.

The simplicial Barratt--Eccles operad $\rE\sym$ is obtained by applying the functor $\EE$ to the operad $\sym$.
We denote its composition by
\[
\bcirc_{\rE\sym} \colon
\rE\sym(r) \times \rE\sym(s_1) \times\dots\times \rE\sym(s_r) \to
\rE\sym(s_1+\dots+s_r).
\]

The (algebraic) Barratt--Eccles operad $\BE$ is defined by $\BE(r) = \chains\EE\sym_r$ with composition given by
\[
\bcirc_{\BE} \defeq \bcirc_{\EE} \circ \EZ \colon \BE(r) \ot \BE(s_1) \ot\dotsb\ot \BE(s_r) \to \BE(s_1+\dots+s_r).
\]
It is an $E_\infty$-operad.
For any simplicial set $X$, Berger and Fresse explicitly constructed a natural operad morphism
\[
\phi \colon \BE \to \End(\cochains(X))
\]
on its (normalized) cochains.
Since only the existence of this effective construction will be used here, we refer to the original source for details, and to the specialized computer algebra system \texttt{ComCH} for an implementation.

\subsection{May--Steenrod structure on Barratt--Eccles algebras}

To provide the cochains of simplicial sets, or more generally any $\BE$-algebra, with a natural May--Steenrod structure, it suffices to define a $\cyc$-equivariant quasi-isomorphism
\[
\iota \colon \cW \to \BE.
\]
We recall from \cite{medina2021may_st} a closed form formula for one such map which factors through the $\cyc$-module
\[
\cC = \{\chains\EE\cyc_r\}_{r\in\N}.
\]
For $r,n\in\N$, let
\begin{equation*}
	\iota(e_{n}) =
	\begin{cases}
		\displaystyle{\sum_{s_1, \dots, s_m}} \big(0, {s_1}, {s_1+1}, {s_2}, \dots, {s_{m}}, {s_{m}+1} \big) & n = 2m, \\
		\displaystyle{\sum_{s_1, \dots, s_m}} \big(0, 1, {s_1}, {s_1+1}, \dots, {s_{m}}, {s_{m}+1} \big) & n = 2m+1,
	\end{cases}
\end{equation*}
where the sum is over all $s_1, \dots, s_m \in \{0, \dots, r-1\} \cong \cyc_r$.
Please consult \cref{f:small values of psi} for a few examples.

\begin{table}
	\centering
	\resizebox{0.8\columnwidth}{!}{%
\renewcommand{\arraystretch}{1.2}
\begin{tabular}{|c||c|c|c|}
	\hline
	$r$ & $n=2$ & $n=3$ & $n=4$ \\
	\hline
	2 & (0,1,0) & (0,1,0,1) & (0,1,0,1,0) \\
	\hline
	3 & (0,1,2) + (0,2,0) & (0,1,2,0) + (0,1,0,1) & \phantom{+} (0,1,2,0,1) + (0,1,2,1,2) \\
	& & & + (0,2,0,1,2) + (0,2,0,2,0) \\
	\hline
	4 & (0,1,2) + (0,2,3) & (0,1,2,3) + (0,1,3,0) & \phantom{+} (0,1,2,3,0) + (0,1,2,0,1) \\
	& + (0,3,0) & + (0,1,0,1) &
	+ (0,1,2,1,2) + (0,2,3,0,1) \\
	& & & + (0,2,3,1,2) + (0,2,3,2,3) \\
	& & & + (0,3,0,1,2) + (0,3,0,2,3) \\
	& & & + (0,3,0,3,0) \\
	\hline
\end{tabular}
}
\vspace*{3pt}
	\caption{The elements $\psi(e_n)$ for small values of $r$ and $n$.}
	\label{f:small values of psi}
\end{table}

Given a Barratt--Eccles algebra $\phi \colon \BE \to \End(A)$ we will consider $A$ as a May--Steenrod complex with the structure defined by $\phi$ and $\iota$.

\subsection{Main diagram}

Our goal is to construct a natural Cartan relator for $\BE$-algebras, that is to say, a natural $\cyc$-equivariant chain homotopy $H \colon \cW \to \End(A)^\vee$ from $\tau F$ to $G$ for any such $A$.
We will do so by constructing $\cyc_r$-equivariant chain homotopies $K_1$, $K_2$, and $K_3$, and group homomorphisms $f$ and $g$ in the following diagram:
\begin{equation}\label{d:big diagram}
	\begin{tikzcd}
	&[0pt] \BE(r) \arrow[r,"\id\, \ot \,e^{\ot r}"]
	&[20pt] \BE(r) \ot \BE(2)^{\ot r}
	\arrow[d,"\bcirc_\BE"]
	&[-20pt] \\
	& & \BE(2r)
	\arrow[d, "\cdot \tau"] & \\
	\cW(r)
	\arrow[ruu,bend left,"\iota"]
	\arrow[r,"\iota"]
	\arrow[d,"\Delta"']
	\arrow[dr,phantom,"K_3"]
	& \cC(r)
	\arrow[uu,hook]
	\arrow[r,bend left,"\chains\EE f \cdot \tau"]
	\arrow[r,bend right,"\chains\EE g"']
	\arrow[d,"\Delta_{\AW}"]
	\arrow[r,phantom,"K_1"]
	& \BE(2r)
	\arrow[r,"\phi"]
	& \End(A)(2r) \\
	\cW(r)^{\ot 2}
	\arrow[r,"\iota^{\ot 2}"']
	\arrow[dr,bend right,"\iota^{\ot 2}"']
	& \cC(r)^{\ot 2}
	\arrow[dr,phantom,"K_2",shift left=7pt]
	\arrow[d,hook] & & \\
	& \BE(r)^{\ot 2}
	\arrow[r,"e\, \ot\, \id"]
	& \BE(2) \ot \BE(r)^{\ot 2}
	\arrow[uu,"\bcirc_\BE"']
\end{tikzcd}
\end{equation}
where the top and bottom compositions are respectively $\tau F$ and $G$, and $e$ denotes the identity in degree $0$.

\subsection{The chain homotopy $K_1$}

Let us focus on the following part of Diagram \eqref{d:big diagram}:
\[
\begin{tikzcd}[column sep=large]
	\cC(r)
	\arrow[r,bend left,"\chains\EE f \cdot \tau"]
	\arrow[r,bend right,"\chains\EE g"']
	\arrow[r,phantom,"K_1"]
	& \BE(2r).
\end{tikzcd}
\]
For $r \in \N$, let $f,g \colon \cyc_r \to \sym_{2r}$ be the following group homomorphisms:
\[
\begin{split}
	&f \colon \cyc_r \hookrightarrow \sym_r \xra{\id \times e^r} \sym_r \times \sym_2^r \xra{\bcirc_{\sym}} \sym_{2r}\,, \\
	&g \colon \cyc_r \hookrightarrow \sym_r \xra{e \times \rD} \sym_2 \times \sym_r \times \sym_r \xra{\bcirc_{\sym}} \sym_{2r}\,,
\end{split}
\]
where $e$ denotes the identity element.
Explicitly,
\begin{align*}
	f(\rho) &= (1,3,\dots,2r-1)(2,4,\dots,2r), \\
	g(\rho) &= (1,2,\dots,r)(r+1,r+2,\dots,2r),
\end{align*}
where $\rho = (1,2,\dots,r)$.
Please notice the following identity involving the reordering $\tau$ of \cref{ss:reordering}:
\begin{equation}\label{eq:conjugation of little maps}
	f(\rho) \tau = \tau g(\rho).
\end{equation}

The map defined by
\[
K_1(a_0,\dots,a_n) =
\sum_{i=0}^n \ (-1)^i (f(a_0) \tau, \dots, f(a_i) \tau, g(a_i), \dots, g(a_n))
\]
can be directly seen to be a chain homotopy from $(\chains \EE f) \cdot \tau$ to $\chains \EE g$, and we can use \cref{eq:conjugation of little maps} together with the fact that $\rho$ acts on $\BE(2r)$ multiplying on the right by $g(\rho)$ to observe that $K_1$ is $\cyc_r$-equivariant.

\subsection{The chain homotopy $K_2$}

Let us now focus on the following part of Diagram~\eqref{d:big diagram}:
\[
\begin{tikzcd}
	\cC(r)
	\arrow[r,"\chains\EE g"]
	\arrow[d,"\Delta_{\AW}"']
	\arrow[ddr,phantom,"K_2",shift left=0pt]
	&[15pt] \BE(2r)
	\\
	\cC(r)^{\ot 2}
	\arrow[d,hook] & \\
	\BE(r)^{\ot 2}
	\arrow[r,"e\, \ot\, \id"]
	& \BE(2) \ot \BE(r)^{\ot 2}.
	\arrow[uu,"\bcirc_\BE"']
\end{tikzcd}
\]
Since the functor $\EE$ is symmetric monoidal, the top and bottom compositions of this diagram agree with those in:
\[
\begin{tikzcd}
	\chains\EE\cyc_r
	\rar[hook]
	&[-10pt]
	\chains\EE\sym_r
	\rar["\chains\EE\rD"]
	\arrow[dr,in=180,out=-90,"\Delta_{\AW}"'] &
	\chains\EE\sym_r^{\times 2}
	\rar["\chains\EE \, e \times \id"]
	\dar["\AW"'] &[0pt]
	\chains\EE(\sym_2 \times \sym_r^{\times 2})
	\rar["\chains\bcirc_{\EE\!\sym}"]
	\dar["\AW"']
	\drar["\id",out=-10,in=150] &
	\chains\EE\sym_{2r} \\ & &
	\chains\EE\sym_r^{\ \ot 2}
	\rar["e \,\ot\, \id"] &
	\chains\EE\sym_2 \ot \chains\EE\sym_r^{\ \ot 2}
	\rar["\EZ"]
	\uar["\Shi",phantom,shift right=30pt] &
	\chains\EE(\sym_2 \times \sym_r^{\times 2})
	\uar["\chains\bcirc_{\EE\!\sym}"']
\end{tikzcd}
\]
Therefore, it suffices to define
\[
K_2 = \chains\bcirc_{\EE\sym} \circ \Shi \circ \chains\EE(e \times \id) \circ \chains\EE\rD
\]
as the desired $\cyc_r$-equivariant chain homotopy.

\subsection{The chain homotopy $K_3$} \label{ss:coproduct}

We now study the failure of the $\cyc$-equivariant quasi-isomorphism $\iota$ to preserve coalgebra structures.
More precisely, we will construct a $\cyc$-equivariant chain homotopy $K_3$ making the following diagram commute:
\begin{equation}\label{d:coproducts}
	\begin{tikzcd}[column sep=large]
			\cW \arrow[r,"\iota"] \arrow[d,"\Delta"'] \arrow[dr,phantom,"K_3"]&
			\cC \arrow[d,"\Delta_{\AW}"] \\
			\cW^{\ot 2} \arrow[r,"\iota^{\ot 2}"'] &
			\cC^{\ot 2}.
		\end{tikzcd}
\end{equation}

Using the bijection $\cyc_r \cong \set{0,\dots,r-1}$, we denote by $\alpha(a_1,\dots,a_\ell)$ the number of increasing consecutive pairs $a_i < a_{i+1}$ in a basis element of $\cC(r)$.
For two such elements $(s_1,\dots,s_j)$ and $(t_1,\dots,t_k)$ we denote respectively by $\varphi(0,s_1,\dots,s_j;t_1,\dots,t_k)$ and $\varphi(1,s_1,\dots,s_j;t_1,\dots,t_k)$ the following expressions in $\cC(r) \ot \cC(r)$:
\[
\alpha(s_1,\dots,s_j,t_1) \cdot
(0,s_1,s_1+1,\dots,s_j,s_j+1) \ot
(t_1,t_1+1,\dots,t_k,t_k+1),
\]
and
\[
- \alpha(s_1-1,\dots,s_j-1,t_1-1) \cdot
(0,1,s_1,s_1+1,\dots,s_j,s_j+1)\otimes (t_1,t_1+1,\dots,t_k,t_k+1).
\]
Then, the chain homotopy $K_3$ is defined by
\[
\begin{split}
	K_3(e_{2i})   &= \sum \, \varphi(0,s_1,\dots,s_j;t_1,\dots,t_k), \\
	K_3(e_{2i+1}) &= \sum \, \varphi(1,s_1,\dots,s_j;t_1,\dots,t_k),
\end{split}
\]
where the sums are taken over all $s_1,\dots,s_j,t_1,\dots,t_k \in \cyc_r$ with $i+1 = j+k$.
Please consult \cref{f:small values of K} for a few examples.

The verification that $K_3$ is a $\cyc_r$-equivariant chain homotopy from $(\iota\otimes \iota)\circ \Delta$ to $\Delta_{\AW}\circ \iota$ is done through a somewhat arduous computation presented in \cref{s:postponed}.

\begin{table}
	\centering
	\resizebox{0.8\columnwidth}{!}{%
\renewcommand{\arraystretch}{1.2}
$\begin{array}{|c||c|c|}
	\hline
	n & r=3 & r=4 \\
	\hline
	2 &   & \phantom{+}(0,1,2)\otimes (2,3) \\
	&(0,1,2)\otimes(0,1) &+ (0,1,2)\otimes(3,0) \\
	&&+ (0,2,3)\otimes (3,0) \\
	\hline
	3 &   &\phantom{+} (0,1,2,3)\otimes(3,0) \\
	&(0,1,2,0)\otimes (0,1)&+ (0,1,2,3)\otimes(0,1) \\
	&& + (0,1,3,0)\otimes (0,1) \\
	\hline
	4 & \phantom{+} (0,1,2,0,1)\otimes (1,2) + (0,1,2,0,1)\otimes (2,0)& \\
	&+ (0,1,2,1,2)\otimes(2,0) + (0,2,0,1,2)\otimes(2,0) & \text{33 terms} \\
	&+ (0,1,2)\otimes (2,0,1,2) + (0,1,2)\otimes (2,0,2,0) &\\
	\hline
	5 & \phantom{+}(0,1,2,0,1,2)\otimes (2,0) + (0,1,2,0,1,2)\otimes (0,1)& \\
	& + (0,1,2,0,2,0)\otimes(0,1) + (0,1,0,1,2,0)\otimes(0,1)& \text{33 terms}\\
	& + (0,1,2,0)\otimes (0,1,2,0) + (0,1,2,0)\otimes (0,1,0,1)& \\
	\hline
\end{array}$
}
\vspace*{5pt}

	\caption{The elements $K_3(e_n)$ for small values of $r$ and $n$. For $r=2$ or $n<2$ all vanish. Notice that the indices are flipped with respect to \cref{f:small values of psi}.}
	\label{f:small values of K}
\end{table}

\subsection{Main construction}

Bringing these constructions together using Diagram~\eqref{d:big diagram} we obtain the following.

\begin{theorem*}
	Let $A$ be a Barratt--Eccles algebra with structure map $\phi$.
	Then,
	\[
	H = \phi \circ K_1 \circ \iota \ +\ \phi \circ K_2 \circ \iota \ +\ \phi \circ \bcirc_{\BE} \circ (e \ot \id)\circ K_3
	\]
	is a natural Cartan relator for $A$.
	Therefore, for $i \in \N$ and mod $p$ cocycles $a,b \in A$
	\[
	\bd H(e_i)(a^{\ot p} \ot b^{\ot p}) = C_\psi^p(i)(a, b)
	\]
	where $C_\psi^p(i)(a, b)$ is the lift of the Cartan formula defined in {\rm \cref{ss:cartan_coboundary}}.
\end{theorem*}

\section{Postponed proof}\label{s:postponed}

In this section we prove that $K_3 \colon \cW(r) \to \cC(r) \ot \cC(r)$ is a $\cyc$-equivariant chain homotopy between $(\iota \ot \iota) \circ \Delta$ and $\Delta_{\AW} \circ \iota$.
In \cref{ss:equivariant homotopy general} we give a general procedure to construct equivariant homotopies, which we use to define a $\cyc$-equivariant chain homotopy $K$ between $(\iota\ot \iota) \circ \Delta$ and $\Delta_{\AW} \circ \iota$.
In \cref{ss:closed formula for K} we give a closed form formula for $K$, and in \cref{ss:comparins K and K3} we show that this formula agrees with the one defining $K_3$.

\subsection{Equivariant homotopy construction}\label{ss:equivariant homotopy general}

\begin{lemma}
	Let $G$ be a group, $C$ a bounded-below graded chain complex of free $\Z[G]$-modules with a basis, and $D$ a graded chain complex of $\Z[G]$-modules together with a $\Z$-linear endomorphism $\eta$ satisfying for some chain endomorphism $\varepsilon$ that:
	\begin{enumerate}
		\item $\bd \circ \, \eta + \eta \circ \bd = \id_C - \varepsilon$,
		\item $\varepsilon \circ \eta = 0$.
	\end{enumerate}
	Then, for any pair of $\Z[G]$-linear chain maps $\mu,\nu \colon C \to D$ with $\varepsilon \circ (\mu-\nu) = 0$, the $\Z[G]$-linear map $K \colon C \to D$ recursively defined on basis elements by
	\[
	K(b) = \eta \circ (\mu - \nu - K \circ \bd)(b),
	\]
	satisfies
	\[
	\bd \circ \, K + K \circ \bd = \mu - \nu.
	\]
\end{lemma}

\begin{proof}
	We will use an induction argument.
	Suppressing composition symbols we have
	\begin{align*}
		(\bd K+ 'bd)(b)
		&= \bd \eta(\mu-\nu-K\bd)(b) + K\bd(b) \\
		&= (\id_C - \varepsilon - \eta\bd)(\mu-\nu - K\bd)(b) + K\bd(b).
	\end{align*}
	If $b$ is a cycle we are done since $\varepsilon(\mu-\nu) = 0$ by assumption, in particular if $b$ is of lowest degree. For an arbitrary $b$, since $\varepsilon \eta = 0$ we have $\varepsilon K \bd(b) \defeq \varepsilon \eta(\mu-\nu-K)(\bd b) = 0$, so $\varepsilon(\mu-\nu-K\partial) = 0$.
	Additionally, by the induction assumption $\bd K (\bd b) = (\mu-\nu-K\bd)(\bd b) = (\mu-\nu)(\bd b)$ we have $\eta \bd (\mu-\nu - K\bd)(b) = 0$.
	Therefore, we are left with
	\begin{align*}
		(\bd K+ K\bd)(b) &=
		\id_C(\mu-\nu-K\bd)(b)+K\bd(b) \\&=
		(\mu-\nu)(b)
	\end{align*}
	as desired.
\end{proof}

We can apply this lemma to obtain a $\cyc$-equivariant homotopy $K$ from $\mu = (\iota \ot \iota) \circ \Delta$ to $\nu = \Delta_{\AW} \circ \iota$, using the endomorphisms of $\cC(r) \ot \cC(r)$ given by
\[
\eta \big((s_0,\dots,s_j) \ot (t_0,\dots,t_k)\big) = (0,s_0,\dots,s_j) \ot (t_0,\dots,t_k),
\]
and
\[
\varepsilon(s \ot t) =
\begin{cases}
	\rho^0 \ot t & \text{if } \deg(s \ot t) = 0, \\
	\hfil0 & \text{otherwise}.
\end{cases}
\]

\subsection{A closed formula for $K$}\label{ss:closed formula for K}

Using the bijection $\cyc_r\cong\{0,1,\dots,r-1\}$ the group $\cyc_r$ receives a total order.
Consider additionally the cyclic order on $\cyc_r$, which is a ternary relation given on elements $a,b,c\in \{0,1,\dots,r-1\}$ as follows: $a\prec b\prec c$
if there are representatives $\bar{a},\bar{b},\bar{c}$ of the classes $[a],[b],[c]$ in $\Z_r$ such that $\bar{a}<\bar{b}<\bar{c}<\bar{a}+r$.

For $(s_1,\dots,s_j)$ in $\EE\cyc_r$ we use the following notation:
\[
s_q^- =
\begin{cases}
	s_{q-1}+1 & \text{if } q>1, \\
	0 & \text{if } q=1,
\end{cases}
\]
and for each $1\leq q\leq j$, write $\theta_q(0,s_1,\dots,s_j;t_1,\dots,t_k)$ for
\[
(0,s_1,s_1+1,\dots,s_j,s_j+1)\ot (t_1,t_1+1,\dots,t_k,t_k+1)
\]
if $s_q^- \prec s_q \prec t_1$ and $0$ otherwise, and $\theta_q(1,s_1,\dots,s_j;t_1,\dots,t_k)$ for
\[
(0,1,s_1,s_1+1,\dots,s_j,s_j+1)\ot (t_1,t_1+1,\dots,t_k,t_k+1)
\]
if $s_q^--1 \prec s_q-1 \prec t_1-1$ and $0$ otherwise.

\begin{lemma*}
	The map $K$ is defined by the following expressions
	\begin{align} \label{eq:homotopyK'1}
		K(e_{2i}) &= \sum_{i+1 = j+k} \sum_{q=1}^j \
		\theta_q(0,s_1,\dots,s_j;t_1,\dots,t_k), \\ \label{eq:homotopyK'2}
		K(e_{2i+1}) &= -\sum_{i+1 = j+k} \sum_{q=1}^j \
		\theta_q(1,s_1,\dots,s_j;t_1,\dots,t_k),
	\end{align}
	where the sum is taken over all $s_1,\dots,s_j,t_1,\dots,t_k\in \{0,1,\dots,r-1\}$.
\end{lemma*}

\begin{proof} 
 	Let us proceed by induction.
	In degree $0$ we have that $\mu(e_0) = \nu(e_0) = 0\ot 0$, therefore we have that
	\[K(e_0) = \eta \circ (\mu-\nu-K \circ \partial)(e_0) = 0.\]
	which agrees with \eqref{eq:homotopyK'1} for $i=0$.
	Assume now that we have proven that the homotopy has that form up to degree $2i$ with $i\geq 0$.
	Then
	\[K(e_{2i+1}) = \eta \circ (\mu-\nu-K \circ \partial)(e_{2i+1}).\]
	Since $\eta \circ \nu(e_{2i+1}) = 0$ and $\eta \circ \mu(e_{2i+1}) = 0$, we are left with
	\[-\eta \circ K \circ \partial(e_{2i+1}) = -\eta \circ K (\rho e_{2i+1}) + \eta \circ K (e_{2i+1}).
	\]
	The second summand vanishes, because all terms in $K(e_{2i+1})$ start with $0$, and the first summand is precisely the formula \eqref{eq:homotopyK'2}.

	Assume now that we have proven that the homotopy has that form up to degree $2i-1$ with $i>0$.
	Then
	\[K(e_{2i}) = \eta \circ (\mu-\nu-K \circ \partial)(e_{2i}).\]
	Again, $\eta \circ \nu(e_{2i}) = 0$, but now a careful look reveals that
	\begin{equation}\label{eq:hK1}\eta \circ \mu(e_{2i}) = \sum_{j+k=i+1}\theta_1(0,s_1,\dots,s_j;t_1,\dots,t_j).
	\end{equation}
	Regarding the summand $\eta \circ K \circ \partial(e_{2i})$, we have
	\[\eta(N\theta_q(1,s_1,\dots,s_j,t_1,\dots,t_k)) = \sum_{s_1' = 1}^{r-1}\theta_{q+1}(0,s_1',s_2',\dots,s_{j+1}';t_1,\dots,t_k),\]
	where $s_{i+1}' = s_i$ for $i\geq 1$, and therefore
	\begin{align*} \label{eq:hK2}
		-\eta \circ K \circ \partial(e_{2i})
		&= -\eta\left(-\sum_{j+k = i}\sum_{q=1}^j N\theta_q(1,s_1,\dots,s_j;t_1,\dots,t_k)\right) \\
		&= \sum_{j+k = i+1}\sum_{q=2}^{j+1}\theta_q(0,s_1,\dots,s_{j+1};t_1,\dots,t_k).
	\end{align*}
	Putting together this last formula and \eqref{eq:hK1}, we obtain the formula \eqref{eq:homotopyK'1}.
\end{proof}

\subsection{Comparing $K$ and $K_3$}\label{ss:comparins K and K3}

The agreement of $K$ and $K_3$ follows from the following.

\begin{lemma*}
	For any $(s_1,\dots,s_j)$ and $(t_1,\dots,t_k)$ in $\EE\cyc_r$ we have:
	\begin{align*}
		\varphi(0,s_1,\dots,s_j;t_1,\dots,t_k) &= \sum_{q=1}^j \theta_q(0,s_1,\dots,s_j;t_1,\dots,t_k) \\
		\varphi(1,s_1,\dots,s_j;t_1,\dots,t_k) &= -\sum_{q=1}^j \theta_q(1,s_1,\dots,s_j;t_1,\dots,t_k).
	\end{align*}
\end{lemma*}

\begin{proof}
	Observe first that if $s_{i+1} = s_i+1$ for some $i$, or $s_1 = 0$ (in the first case) or $s_1=1$ (in the second case) we get degenerate summands, so we do not treat these cases in the following computations.
	For a number $a\in \{0,1,\dots,r-1\}$, we will write $[a]$ for the class of $a$ in $\Z_{r}$.

	\vspace*{5pt}\noindent\textit{First case}.
	Let $(s_1',\dots,s'_j)$ be the only sequence of integers such that:
	\begin{enumerate}
		\item $[s'_i] = [s_i]$,
		\item $0< s_1'< r$,
		\item $1 < s_{i+1}' - s_i' \leq r$.
	\end{enumerate}
	Then $\theta_q(0,s_1,\dots,s_j;t_1,\dots,t_k)$ will vanish for $q>1$ if and only if the interval $(s_{q-1}',s_q']$ contains a representative of $[t_1]$, and will vanish for $q=1$ if and only if the interval $[0,s_1']$ contains a representative of $[t_1]$.
	Therefore the number of vanishing summands in $\sum_{q=1}^j\theta_q(0,s_1,\dots,s_j;t_1,\dots,t_k)$ equals the number $Q$ of representatives of $[t_1]$ that lie in the interval $[0,s_j']$, which in turn is computed by
	\[
	Q =
	\begin{cases}
		\card\set{s_i \geq s_{i+1}} & \text{if } t_1 > s_j, \\
		\card\set{s_i \geq s_{i+1}} + 1 & \text{if } t_1 \leq s_j.
	\end{cases}
	\]
	Hence, the number of non-vanishing summands in $\sum_{q=1}^j\theta_q(0,s_1,\dots,s_j;t_1,\dots,t_k)$ is $j-Q$, which equals $\alpha(s_1,\dots,s_j,t_1)$.

	\vspace*{5pt}\noindent\textit{Second case}.
	It follows from the definition of $\theta_q$ that the number of non-vanishing summands in $\sum_{q=1}^j \theta_q(1,s_1,\dots,s_j;t_1,\dots,t_k)$ equals the number of non-vanishing summands in
	\[
	\sum_{q=1}^j \theta_q(0,s_1-1,\dots,s_j-1;t_1-1,\dots,t_k-1),
	\]
	which by the previous case equals $\alpha(s_1-1,\dots,s_j-1,t_1-1)$.
\end{proof}
	\sloppy
	\printbibliography
\end{document}